\documentclass[a4paper,11pt,reqno]{amsart}
\usepackage{amssymb,amsmath,mathrsfs,graphics,graphicx,mathptm,float,latexsym,pstricks}
\usepackage[T1]{fontenc}
\usepackage[english,francais]{babel}
\usepackage[all]{xy}
\theoremstyle{plain}
\newtheorem{thm}{Theorem}[section]
\newtheorem{pro}[thm]{Proposition}
\newtheorem{lem}[thm]{Lemma}
\newtheorem{cor}[thm]{Corollary}
\newtheorem{theoalph}{Theorem}

\theoremstyle{definition}
\newtheorem{definition}[thm]{Definition}
\newtheorem{example}[thm]{Example}
\newtheorem{rem}[thm]{Remark}

\newtheorem{nota}[thm]{Notation}
\def\og{\leavevmode\raise.3ex\hbox{$\scriptscriptstyle\langle\!\langle$~}}
\def\fg{\leavevmode\raise.3ex\hbox{~$\!\scriptscriptstyle\,\rangle\!\rangle$}}
\def\transp #1{\vphantom{#1}^{\mathrm t}\! {#1}}
\addtocounter{section}{0}             
\numberwithin{equation}{section}       

\def\Deter{\mathcal{D}_{4,4}}
\def\Bir{\mathrm{Bir}}
\def\OO{\mathcal{O}}
\def\PP{\mathbb{P}}
\def\CC{\mathbb{C}}

\def\PxPxP{{\PP_1 \times \PP_3 \times \PP'_3}}
\def\PPT{{\widetilde{\PP}_3}}
\def\planT{{\widetilde{\PP}_2}}
\def\II{\mathscr{I}} 
\def\HH{\mathrm{H}} 
\def\hh{\mathrm{h}} 
\def\IGPT{\II_{\Gamma \mid \PPT}}
\def\IGP{\II_{\overline{\Gamma}}}
\def\sheafExt{{\mathscr{E}\!\!xt}} 
\def\sheafHom{{\mathscr{H}\!\!om}} 
\def\Homdroit{\mathrm{Hom}} 
\def\Cun{C_1}
\def\Cdeux{C_2}
\def\SingF{F^\sigma}

\begin{document}
\selectlanguage{english}
\title{Quarto-quartic birational maps of $\PP_3(\CC)$}
\author{Julie \textsc{D{\'e}serti}}
\email{deserti@math.univ-paris-diderot.fr}
\author{Fr{\'e}d{\'e}ric \textsc{Han}}
\email{frederic.han@imj-prg.fr}
\address{Universit{\'e} Paris Diderot, Sorbonne Paris Cit{\'e}, Institut de Math{\'e}matiques de
Jussieu-Paris Rive Gauche, UMR $7586$, CNRS, Sorbonne Universit{\'e}s, UPMC Univ Paris $06$,
F-$75013$ Paris, France.}
\maketitle
\begin{abstract}
We construct a determinantal family of quarto-quartic transformations of a complex
projective space of dimension $3$  from trigonal curves
of degree $8$ and genus $5$. Moreover we show that the variety of $(4,4)$-birational maps
of  $\PP_3(\CC)$ has at least four irreducible components  and describe three of
them.
\bigskip
\noindent\emph{$2010$ Mathematics Subject Classification. --- $14E05$, $14E07$.}
\end{abstract}
\section{Introduction}
Let $\PP_3$ and $\PP'_3$ be two complex projective spaces of dimension
$3$. Denote by $\Bir_d(\PP_3,\PP'_3)$ the set of birational
maps from $\PP_3$ to $\PP'_3$ defined by a linear system of degree $d$ with base
locus of codimension at least $2$. A rational map $\phi\colon \PP_3 \dashrightarrow \PP'_3$ is
said to have bidegree $(d_1,d_2)$ if it is an element of
$\Bir_{d_1}(\PP_3,\PP'_3)$ such that its inverse is in $\Bir_{d_2}(\PP'_3,\PP_3)$.
Let $\Bir_{d_1,d_2}(\PP_3,\PP'_3)$ be the quasi projective variety
(\cite{Pan:these}) of birational maps from $\PP_3$ to $\PP'_3$
of bidegree $(d_1,d_2)$.
It is obvious that a general surface in a linear system of degree $d$
defining an element of $\Bir_d(\PP_3,\PP'_3)$ must be rational, but it should
have a strong impact for $d>3$.  Our first motivations to study
$\Bir_{4,4}(\PP_3,\PP'_3)$ were to understand some typical  differences between
$d\leq 3$ and $d>3$.
For $d\leq 3$ a large amount of examples was already known a long time ago
(\cite{Hudson}) and recent works such as \cite{PanRongaVust} or
\cite{DesertiHan} were more involved with classification problems. But for $d>3$ only few
examples are known. Basically, for $d>3$, the classical examples are built by lifting a
birational transformation of a projective space of lower dimension (\cite[\S
7.2.3]{CAG}). In particular, for $d\geq 2$ monoidal linear systems give the de
Jonqui{\`e}res family $\mathcal{J}_{d,d}$ (\emph{see} \cite{Pan:these}, \cite{Pan:bresil}). Another
classical family (denoted by $\mathcal{R}_{d,d}$) can be constructed  with surfaces of degree
$d$ with a line of multiplicity $d-1$. As expected from the components of
$\Bir_{3,3}(\PP_3,\PP'_3)$, we show that the closure of $\mathcal{R}_{4,4}$ is an irreducible
component of $\Bir_{4,4}(\PP_3,\PP'_3)$ (Proposition~\ref{compR44}).
A first surprise arises from normality properties. Indeed, in
$\Bir_{3}(\PP_3,\PP'_3)$ linear systems without a normal surface were
exceptional and similar to elements of $\mathcal{R}_{3,3}$, but it turns out that any
element of $\Bir_{4,4}(\PP_3,\PP'_3)$ with a linear system containing a normal
surface must be in $\mathcal{J}_{4,4}$ (Lemma~\ref{lem:normalrational}). So the
closure of $\mathcal{J}_{4,4}$ is an irreducible component of
$\Bir_{4,4}(\PP_3,\PP'_3)$ (Corollary \ref{compJ44}) and most
of the work occurs with non normal surfaces. On the other hand $\mathcal{J}_{3,3}$ is not
an irreducible component of $\Bir_{3,3}(\PP_3,\PP'_3)$  because it is in the
boundary of the classical cubo-cubic determinantal family (\cite[\S 4.2.2]{DesertiHan}).
The classical  cubo-cubic transformations (\cite[7.2.2]{CAG}) are built from
Arithmetically Cohen Macau\-lay curves of degree $6$ and genus $3$. This family is so particular
(\cite{Katz}) that it was unexpected to find in $\Bir_{4,4}(\PP_3,\PP'_3)$
another ACM family. Hence most of the present work is about the construction of
the ACM family $\Deter$ and its geometric
properties. We first define elements of $\Deter$ via a direct construction of
the resolution of their base ideal. This resolution is  non generic among the maps of the
following type $$\OO_{\PP_3}^2(-H)\oplus \OO_{\PP_3}(-2H) \longrightarrow \OO^4_{\PP_3}$$
and we provide an explicit construction in \S~\ref{explicit}. Then we get from
Corollary~\ref{cor:descrd44} and Proposition~\ref{propHH} the following geometric
description of $\Deter$.
\begin{theoalph}\label{thm:deter}
Let $\Gamma$ be a trigonal curve of genus $5$, and let $\overline{\Gamma}$ be its
embedding in $\PP_3$ by a general linear system $|\OO_\Gamma(H)|$ of degree
$8$. Then
\begin{itemize}
\item the quartic surfaces containing $\overline{\Gamma}$ and singular along its unique
  $5$-secant line give an element $\phi_{\overline{\Gamma}}$ of $\Deter$;
\item the inverse of $\phi_{\overline{\Gamma}}$ is obtained by the same construction
  with the same trigonal curve $\Gamma$ but embedded in $\PP'_3$ with $|\omega^{\otimes
    2}_\Gamma(-H)|$.
\end{itemize}
\end{theoalph}
\noindent We also geometrically describe the $P$-locus of $\phi_{\overline{\Gamma}}$
(\emph{see} \S~\ref{subsec:Plocus}).
\medskip
Finally we exhibit an other family $\mathcal{C}_{4,4}$ constructed with linear
systems of quartic surfaces with a double conic.
\medskip
While studying components of $\Bir_{4,4}(\PP_3,\PP'_3)$ it turns out that the
following invariant allows to distinguish all the components of $\Bir_{4,4}(\PP_3,\PP'_3)$ we
find. Again, it is different in
$\Bir_3(\PP_3,\PP'_3)$ where this invariant is not meaningful  because in most
examples it is $1$ (\cite[Prop. 2.6]{DesertiHan}).
\begin{definition}{\cite[Chapter IX]{Hudson}}\label{genus}
Let $\phi$ be an element of  $\Bir_d(\PP_3,\PP'_3)$, and let $H$ $($resp. $H')$
be a general hyperplane of $\PP_3$ $($resp. $\PP'_3)$.
The \textbf{\textit{genus of $\phi$}} is the geometric genus of the curve
$H\cap\phi^{-1}(H')$.
\end{definition}
\begin{theoalph}\label{thm:main}
The quasi projective variety $\mathrm{Bir}_{4,4}(\PP_3,\PP_3')$ has at least four irreducible
components. Three of them are the closure of the families $\Deter$, $\mathcal{J}_{4,4}$ and
$\mathcal{R}_{4,4}$. The family $\mathcal{C}_{4,4}$
is in another irreducible component.
Furthermore we have
\[
\begin{array}[c]{cccc}
  \dim(\mathcal{R}_{4,4})=37, &\dim(\mathcal{C}_{4,4})=37,& \dim(\Deter)=46,
  &\dim(\mathcal{J}_{4,4})=54, \\
  g(\phi_{\mathcal{R}_{4,4}})=0, &  g(\phi_{\mathcal{C}_{4,4}})=1, &
  g(\phi_{\Deter})=2, & g(\phi_{\mathcal{J}_{4,4}})=3,
\end{array}
\]
where $g(\phi_{X})$ denotes the genus of an element $\phi_X$ general in the family $X$.
\end{theoalph}
\subsection*{Acknowledgment} The authors would like to thank the referee for his helpful comments.
\tableofcontents
\section{Definitions and first properties}\label{NOTATIONS}
Let $\phi\colon\PP_3\dashrightarrow\PP_3'$ be a rational map given, for some choice of coordinates, by
\[
(z_0:z_1:z_2:z_3)\dashrightarrow\big(\phi_0(z_0,z_1,z_2,z_3):\phi_1(z_0,z_1,z_2,z_3):\phi_2(z_0,z_1,z_2,z_3):\phi_3(z_0,z_1,z_2,z_3)\big)
\]
where the $\phi_i$'s are homogeneous polynomials of the same degree $d$
without common factors. The \textbf{\textit{indeterminacy set}} of $\phi$ is the set of the common
zeros of the $\phi_i$'s. Denote by $\II_\phi$ the ideal generated by the
$\phi_i$'s. Let $F_\phi$ be the scheme defined by $\II_\phi$; it is called
\textbf{\textit{base locus}} or \textbf{\textit{base scheme}} of
$\phi$. If~$\dim F_\phi=1$ then define $F_\phi^1$ as the maximal
subscheme of $F_\phi$ of dimension $1$ without isolated point and without
embedded point. The \textbf{\textit{$P$-locus}} of $\phi$ is the union of
irreducible hypersurfaces that are  blown down to subvarieties of codimension
at least $2$ (\emph{see} \cite[\S~7.1.4]{CAG}).
\begin{lem}\label{singS}
Let $\phi\colon\PP_3\dashrightarrow\PP_3'$ be a $(4,d)$ birational
map. Let $H'$ be a general hyperplane  of $\PP_3'$. Let $\gamma$ be the
reduced scheme with support the singular locus of $\phi^{-1}(H')$ and consider the
secant variety
\[
\mathrm{Sec}(\gamma)=\overline{\bigcup_{\stackrel{p,\,q\in\gamma}{p\not=q}}\delta_{p,q}}
\subset \PP_3
\]
where $\delta_{p,q}$ denotes the line through $p$ and $q$. Then the codimension
of $\mathrm{Sec}(\gamma)$ in $\PP_3$ is strictly positive.
\end{lem}
\begin{proof}
Let $p$, $q$ be two distinct points of $\gamma$. The line $\delta_{p,q}$ passing through
$p$ and $q$ intersects $F_\phi^1$ with length at least $4$ since we have the
inclusion of ideals
$\II_{F_\phi^1}\subset\II_\gamma^2$; the line $\delta_{p,q}$ is thus blown
down by $\vert\II_{F^1_\phi}(4H)\vert$.

But $\phi$ is dominant, so a general point of $\PP_3$ is not in the $P$-locus of
$\phi$. Hence $\mathrm{codim}_{\PP_3}(\mathrm{Sec}(\gamma))>0$.
\end{proof}
\begin{cor}\label{majlieusing}
Let $\phi\colon\PP_3\dashrightarrow\PP_3'$ be a $(4,d)$ birational
map. Let $H'$ be a general hyperplane  of~$\PP_3'$. Let $\gamma$ be the
reduced scheme defined by the singular locus of $\phi^{-1}(H')$. If
$\dim\gamma=1$, then the $1$-dimensional part of $\gamma$ is either a line, or a
conic $($possibly degenerated$)$, or $3$ lines through a fixed point.
\end{cor}
\begin{proof}
A general hyperplane section of the quartic surface   $\phi^{-1}(H')$ is an
irreducible  reduced plane quartic curve. So it has at most $3$ singular
points hence $\deg (\gamma) \leq 3$ when $\dim(\gamma)=1$. Moreover $\gamma$ can't contain a plane curve
of degree $3$ because $\phi^{-1}(H')$ is an irreducible quartic surface. From
Lemma~\ref{singS}, neither a space cubic curve, nor two disjoint lines, nor a line
and a smooth conic can be in $\gamma$. Hence only a line, a conic or 3 lines
through a fixed point are possible.
\end{proof}
Unfortunately we have not studied\footnote{We
don't expect that this example is related to one of the $4$ components described
in Th.~B; it seems to us that it should give one more component of
$\mathrm{Bir}_{4,4}(\PP_3,\PP_3')$.} families with $3$ lines through a point
because we first missed such examples.
\begin{example}(Loria 1890, \cite[Chapter XIV \S 8.~$T_{4,4}$ Steiner quartics]{Hudson})
 Let $l_0,l_1,l_2$ be lines in $\PP_3$ of ideal $(z_1,z_2)$, $(z_2,z_0)$, $(z_0,z_1)$ and $O_1,O_2,O_3$
 be three general points in the plane $z_3=0$. For $1 \leq j \leq 3$, let $q_j$ be the quadric cone containing
 the lines $l_0,l_1,l_2$ and the two points of $O_1, O_2,O_3$ distinct from
 $O_j$. Then the following quartics
 $$
(q_1 q_2, q_0 q_2, q_0 q_1, z_0 z_1 z_2 z_3)
 $$
 gives a quarto-quartic birational transformation; the general element of this
 linear system is a Steiner quartic.
\end{example}

\subsection{Families in $\mathrm{Bir}_{p,q}(\PP_3,\PP_3')$ and
  semi-continuity}
\begin{definition}\label{def:semi}
Let $p$, $q$ be integers such that $\mathrm{Bir}_{p,q}(\PP_3,\PP_3')$ is
not empty, and let $S$ be a reduced subscheme of
$\mathrm{Bir}_{p,q}(\PP_3,\PP_3')$. From the natural embedding of
$$
\mathrm{Bir}_{p,q}(\PP_3,\PP_3')
\subset
\PP\left(\Homdroit\left(\HH^0(\OO_{\PP'_3}(1)),\HH^0(\OO_{\PP_3}(p)) \right)\right)
$$
we obtain by pull back from $S$ to $S\times \PP_3$ the tautological sequence
$$
\HH^0(\OO_{\PP'_3}(1)) \otimes \OO_{S\times \PP_3}(-1,0) \longrightarrow \HH^0(\OO_{\PP_3}(p))
\otimes \OO_{S\times \PP_3} .
$$
It gives the following two maps and we can globalize over $S$ the
previous definitions of the base scheme and its singular locus.
\begin{itemize}
  \item Let $F_S$ be the subscheme of $S \times \PP_3$ defined by the
vanishing of the tautological map
$$
\HH^0(\OO_{\PP'_3}(1))\otimes \OO_{S\times \PP_3}(-1,0)
\longrightarrow \OO_{S\times \PP_3}(0,p)
$$
\item Let $\SingF_S$ be the subscheme of $S\times \PP_3$ defined by the vanishing
  of the Jacobian map
$$
\HH^0(\OO_{\PP'_3}(1))\otimes (\HH^0(\OO_{\PP_3}(1)))^\vee \otimes \OO_{S\times\PP_3}(-1,0)
\longrightarrow \OO_{S\times\PP_3}(0,p-1)
$$
\item For an hyperplane $H$ of $\PP_3$ we define the restrictions to $H$ as the
  following intersections in $S\times \PP_3$
  $$F_S\cdot H = F_S \cap ( S \times H),\quad \SingF_S \cdot H = \SingF_S \cap ( S \times H)$$
\end{itemize}
\end{definition}
\begin{cor}\label{dimrelF}
  The morphisms $F_S \to S$ and $\SingF_S \to S$ are projective morphisms, moreover
  all the fibers of $F_S \to S$ have the same dimension.
  \begin{itemize}
  \item   If $q < p^2$ we have for all $\phi$ in $S$, $\dim F_\phi =1$
    (\emph{see} Notation~\ref{notaC1C2}),
  \item   if $q = p^2$ we have for all $\phi$ in $S$, $\dim F_\phi =0$.
  \end{itemize}
\end{cor}
Let us recall the following application of the semi-continuity of the dimension
of the fibers of a coherent sheaf.
\begin{lem}\label{scsfinite}
  Let $f\colon X \to Y$ be a finite morphism of Noetherian schemes over
  $\mathbb{C}$. The map
  $$ Y \to
  \mathbb{Z}, ~ y \mapsto \deg(X_y)$$
  is upper semi-continuous.
\end{lem}
\begin{proof}
  Let $\mathbb{C}(y)$ be the residual field of $Y$ at $y$. Remark that the degree of $X_y$
  is the dimension of the complex vector space $f_*(\OO_X) \otimes \mathbb{C}(y)$ because $f$ is finite.
  By \cite[Th.~III-8.8]{Ha} the sheaf $f_*(\OO_X)$ is coherent on $Y$, so we
  conclude by semi-continuity of the dimension of the fibers of a
  coherent sheaf (\emph{see} \cite[Ex.~III-12.7.2]{Ha}).
\end{proof}
\begin{cor}\label{cor:scs}
  With the previous notation, the following functions are upper semi-continuous
  \begin{enumerate}
  \item $\alpha\colon S \to \mathbb{Z},~ \phi \mapsto \deg(F_\phi)$
  \item $\beta\colon S \to \mathbb{Z},~ \phi \mapsto \dim(\SingF_\phi)$
  \item $\eta\colon \beta^{-1}(\{1\}) \to \mathbb{Z},~ \phi \mapsto \deg(\SingF_\phi)$
  \end{enumerate}
\end{cor}
\begin{proof}
  The semi-continuity of $\beta$ is from Chevalley's theorem (\emph{see} \cite{EGAIV.3} 13.1.5) because
$\SingF_S \to S$ is proper.
Let us prove the semi-continuity of $\alpha$ by showing that for any $n$ in $\mathbb{Z}$
the set $\alpha^{-1}(]-\infty, n])$ is open. If it is not empty, then consider
  an arbitrary element $\psi$ such that $\alpha(\psi) \leq n$. We can chose an hyperplane $H$ of $\PP_3$
  such that $F_\psi\cap H$ is artinian of length
  $\alpha(\psi)=\deg(F_\psi)$. By Corollary~\ref{dimrelF} the projective
  morphism $\pi_S\colon F_S\cdot H \to S$ is
  generically finite. According to Chevalley's theorem the set
  $$U=\{\phi \in S, \dim(\pi_S^{-1}(\{\phi \})) \leq 0 \}$$
  is an open set of $S$ containing $\psi$. Hence over $U$ we have
  a finite morphism $\pi_U\colon~ F_U\cdot H \to U$. Applying Lemma~\ref{scsfinite}
  to $\pi_U$ we have an open set $U'$ such that
  $$
\forall\, \phi \in U',~ \alpha(\phi)=\deg(F_\phi) \leq \mathrm{length}(F_\phi \cap H) \leq \mathrm{length}(F_\psi
\cap H) = \alpha(\psi) \leq n.
  $$
  So for any element $\psi$ of $\alpha^{-1}(]-\infty, n])$ there is an open set
  $U'$ of $S$ such that
  $$\psi \in U',~ U'\subset \alpha^{-1}(]-\infty, n]).$$
  In  conclusion  $\alpha^{-1}(]-\infty, n])$ is open and $\alpha$ is upper
  semi-continuous on $S$.
The proof is the same for $\eta$ because it is also
  a projective morphism such that all its fibers are one dimensional.
\end{proof}
Let us also recall the following
\begin{lem}\label{lem:scsOmeg}
  Let $\pi\colon X \to Y$ be a projective morphism of noetherian schemes, and
  $\Omega_\pi$ be the sheaf of relative differentials, then for all $n \in
  \mathbb{N}$ the set
  $$
  \{x \in X,~ \mathrm{rank}((\Omega_\pi)_x) \geq n\}
  $$
  is closed in $X$, so its image by $\pi$ is closed in $Y$.
\end{lem}
\begin{proof}
  From \cite[Rem.~II-8.9.1]{Ha} the sheaf $\Omega_\pi$ is coherent on $X$. As a result
  the rank of its fiber is upper semi-continuous by \cite[Ex.~III-12.7.2]{Ha}
  and $\{x \in X,~ \mathrm{rank}((\Omega_\pi)_x) \geq n\}$ is closed in $X$. But
  $\pi$ is projective hence the image of this closed set is closed in $Y$.
\end{proof}

\section{Two classical families in $\Bir_{d,d}(\PP_3,\PP'_3)$}\label{sec:jonqetrul}
It is a classical construction to obtain birational transformations over projective
spaces from a factorization through the blow up of a linear space and a birational
transformation between projective spaces of lower dimension
(\cite[Def.~7.2.9]{CAG}). The following one is a typical example of this construction.
\subsection{The Monoidal de Jonqui{\`e}res family $\mathcal{J}_{d,d}$}\label{subsec:jonquieres}
Many equivalent properties are detailed in (\cite{Pan:these} or \cite{Pan:bresil}) so let us choose
the following definition.
\begin{definition}\label{def:Jonquieres}
Let $d$ be a non zero integer, and let $p$ be a point of $\PP_3$. Choose
coordinates such that the ideal of $p$ in $\PP_3$ is
$\II_p=(z_0,z_1,z_2)$. Consider $S_d=z_3P_{d-1}+P_d$ and
$S_{d-1}=z_3Q_{d-2}+Q_{d-1}$ with $P_i$, $Q_i$ in $\CC[z_0,z_1,z_2]$ homogeneous
of degree $i$ such that $\mathrm{gcd}(S_{d-1},S_d)=1$ and $P_{d-1}Q_{d-1}\neq P_d Q_{d-2}$.
The rational map $\phi\colon \PP_3\dashrightarrow \PP'_3$ defined by an
isomorphism $\PP'_3 \simeq  \vert\II_{\phi}(d H)\vert^\vee $ with
\[
  \II_\phi=S_{d-1}\cdot \II_p+(S_d)
\]
is birational of bidegree $(d,d)$. The base scheme $F_\phi$ is  defined by the
monoidal complete intersection $(S_{d-1},S_d)$ with an immerged point at $p$.
Denote by $\mathcal{J}_{d,d}$ the corresponding family of $\Bir_{d,d}(\PP_3, \PP'_3)$.
\end{definition}
\begin{rem}
With the notation of Definition \ref{def:Jonquieres} we can find coordinates on
$\PP'_3$ such that the map~$\phi$ can be written
\[
\left(z_0:z_1:z_2:\frac{P_{d-1}z_3+P_d}{Q_{d-2}z_3+Q_{d-1}}\right)
\]
it is then clear that the inverse of $\phi$
\[
\left(z_0:z_1:z_2:\frac{Q_{d-1}z_3-P_d}{-Q_{d-2}z_3+P_{d-1}}\right)
\]
is also a monoidal de Jonqui{\`e}res element of $\Bir_{d,d}(\PP'_3,\PP_3)$.
\end{rem}
\begin{cor}[\cite{Pan:these}]\label{cor:dimjonq}
The family $\mathcal{J}_{d,d}\subset\mathrm{Bir}_{d,d}(\PP_3,\PP_3')$ is irreducible
of dimension $2d^2+2d+14$.
\end{cor}
\begin{proof}
The choice of $S_{d-1}$ up to factor is $\binom{d}{d-2}+\binom{d+1}{d-1}-1$
dimensional. The choice of $S_{d}$ up to factor with $S_d \notin S_{d-1}\cdot
\CC[z_0,z_1,z_2]$ has dimension $\binom{d+2}{d}+\binom{d+1}{d-1}-3-1$. With the
choices of an automorphism of $\PP'_3$ and of a point $p\in \PP_3$ we have
$\dim \mathcal{J}_{d,d} = \binom{d}{d-2}+2\binom{d+1}{d-1} + \binom{d+2}{d} +13$.
\end{proof}
Let us recall that  $\mathcal{J}_{3,3}$ is not an irreducible component of
$\Bir_{3,3}(\PP_3,\PP'_3)$ (\emph{see} \cite[\S 4.2.2]{DesertiHan}), so the
next Lemma is quite unexpected.
\begin{nota}\label{notaC1C2}
Let $\phi\colon\PP_3\dashrightarrow\PP_3'$ be a birational
transformation of bidegree $(p,q)$. For a general line $l'$ in $\PP'_3$ we
will denote by
\smallskip
\begin{itemize}
\item $\phi^{-1}(l')$ the complete intersection in $\PP_3$ of degree
$p^2$ defined by $l'$.
\smallskip
\item  $\Cun=\phi^{-1}_*(l')\subset \PP_3$ the closure of the image by
$\phi^{-1}$ of points of $l'$ where $\phi^{-1}$ is regular. As a result $\Cun$ is an
irreducible curve of degree $q$ and geometric genus $0$.
\item $\Cdeux$ the residual scheme of $\Cun$ in $\phi^{-1}(l')$. $($Note that $F^1_\phi$ and
  $\Cdeux$ are set theoretically identical but in general $F^1_\phi$ is just a
  subscheme of $\Cdeux$.$)$
\end{itemize}
Furthermore for any scheme $Z$ we denote by $\II_Z$ its ideal.
\end{nota}
\begin{lem}\label{lem:normalrational}
Let $\phi\colon\PP_3\dashrightarrow\PP_3'$ be a $(4,4)$-birational
map. Let $H'$  be a general  hyperplane  of $\PP_3'$ and denote by $S$ the
surface $\phi^{-1}(H')$. If $S$ is normal, then $\Cun$ is a plane rational
quartic and $\phi$ belongs to $\mathcal{J}_{4,4}$.
\end{lem}
\begin{proof}
The curves $\Cun$ and $\Cdeux$ are geometrically linked by quartic surfaces, so the
dualizing sheaf of $\Cun\cup\Cdeux$ is
$\omega_{\Cun\cup\Cdeux}=\mathcal{O}_{\Cun\cup\Cdeux}(4H)$ and we have the
liaison exact sequence
\[
0\longrightarrow\II_{\Cun\cup\Cdeux}(4H)\longrightarrow\II_{\Cdeux}(4H)\longrightarrow
\sheafHom(\OO_{\Cun},\OO_{\Cun\cup\Cdeux}(4H))\longrightarrow 0.
\]
The surface $S$ is normal and contains $F_\phi^1$. Hence $F_\phi^1$
is generically locally complete intersection. So we have the scheme equality $F_\phi^1=\Cdeux$. Thus
$\mathrm{h}^0(\II_{\Cdeux}(4H))=\mathrm{h}^0(\II_{F_\phi^1}(4H))\geq 4$. Therefore from
the liaison exact sequence, the dualizing sheaf
$\omega_{\Cun}=\sheafHom(\OO_{\Cun},\OO_{\Cun\cup\Cdeux}(4H))$ has at least two
independent sections. In other words the arithmetic genus of $\Cun$ is greater or
equal to $2$. By Castelnuovo inequality (cf. \cite[Th.~3.3]{Ha2}) the arithmetic
genus of a non degenerate locally Cohen-Macaulay space curve of
degree $4$ is at most $1$. Hence $\Cun$ is a plane rational
quartic. By liaison, $\Cdeux$ is a complete intersection $(3,4)$.
Moreover it gives the equality
$\sheafHom(\OO_{\Cun},\OO_{\Cun\cup\Cdeux}(4H))=\OO_{\Cun}(H)$, so
the restriction of $\phi$ to $\Cun$ is a pencil of sections of
$\OO_{\Cun}(H)$. But $\phi$ sends $\Cun$ birationally to a line of $\PP'_3$, hence
this pencil must be the lines through a triple point $p$ of the plane curve
$\Cun$. As $S$ has a finite singular locus, the point $p$ is independent of the
choice of $l'$ and it is a triple point for all the quartics in the linear
system $\vert\II_{F_\phi}(4H)\vert$. In conclusion $F_\phi$ is a monoidal  complete
intersection $(3,4)$ with $p$ as immerged point. Besides $\phi$ belongs to
$\mathcal{J}_{4,4}$ because the irreducibility of $S=S_d$ implies $\mathrm{gcd}(S_{d-1},S_d)=1$ and $P_{d-1} Q_{d-1} \neq Q_{d-2} P_d $.
\end{proof}
\begin{cor}\label{compJ44}
The closure of $\mathcal{J}_{4,4}$ is an irreducible component of $\Bir_{4,4}(\PP_3,\PP'_3)$.
\end{cor}

\subsection{The ruled family $\mathcal{R}_{d,d}$}\label{subsec:reglees}
In this paragraph we will prove as a direct application of Lemma~\ref{lem:ruled} the following
\begin{pro}\label{pro:ruled}
Consider a line $\delta$ in $\PP_3$ and an integer $d$ with $d\geq 2$. Let $\Delta_1$, $\ldots$, $\Delta_{d-1}$ be
$d-1$ general disjoint lines that intersect $\delta$
in length one, and let $p_1$, $\ldots$, $p_{d-1}$ be $(d-1)$ general
points in~$\PP_3$. Then the linear system $\big\vert\II_{\delta}^{d-1}\cap\II_{\Delta_1}\cap\ldots\cap
\II_{\Delta_{d-1}}\cap\II_{p_1}\cap\ldots\cap\II_{p_{d-1}}(d\cdot H)\big\vert$
gives a birational map
\[
\PP_3\dashrightarrow\big\vert\II_{\delta}^{d-1}\cap\II_{\Delta_1}\cap\ldots\cap
\II_{\Delta_{d-1}}\cap\II_{p_1}\cap\ldots\cap\II_{p_{d-1}}(d\cdot H)\big\vert^\vee.
\]
Moreover the base scheme is defined by the ideal  $\II_{\delta}^{d-1}\cap\II_{\Delta_1}\cap\ldots\cap
\II_{\Delta_{d-1}}\cap\II_{p_1}\cap\ldots\cap\II_{p_{d-1}}$ and has degree
$\frac{(d+2)(d-1)}{2}$.
\end{pro}
\begin{definition}
  Birational maps of Proposition \ref{pro:ruled} composed with an isomorphism
  $$\big\vert\II_{\delta}^{d-1}\cap\II_{\Delta_1}\cap\ldots\cap
\II_{\Delta_{d-1}}\cap\II_{p_1}\cap\ldots\cap\II_{p_{d-1}}(d\cdot
H)\big\vert^\vee \simeq \PP'_3 $$
form the family $\mathcal{R}_{d,d}\subset \Bir_{d,d}(\PP_3,\PP'_3)$.
\end{definition}
\begin{rem}
Note that surfaces of degree $d$ in such linear systems have a line of
multiplicity $d-1$, so they are ruled. For $d>3$ there may exist ruled surfaces
of degree $d$ without a line of multiplicity $d-1$. Nevertheless we keep the
terminology "ruled"  for $\mathcal{R}_{d,d}$ to emphasize the analogy with the
traditional naming convention when $d=3$.
\end{rem}
\begin{lem}\label{lem:ruled}
With notation of Proposition~\ref{pro:ruled}, the image of $\PP_3$ by the linear system
\[
\big\vert\II_{\delta}^{d-1}\cap\II_{\Delta_1}\cap\ldots\cap\II_{\Delta_{d-1}}(d\cdot H)\big\vert
\]
is a non degenerate $3$-dimensional variety $X_d$ of degree $d$ of $\PP_{d+2}$.
\end{lem}
\begin{rem}
Let us note that $X_d$ is of minimal degree (\emph{see} \cite{EisenbudHarris}).
\end{rem}
\begin{proof}[Proof of Lemma~\ref{lem:ruled}]
Let us choose coordinates such that the ideal of $\delta$ in $\PP_3$ is $\II_\delta=(z_0,z_1)$. Let
$\II_{\Delta_j}$ be $(\ell_j,\ell'_j)$ with
\smallskip
\begin{itemize}
\item $\mathrm{gcd}(\ell_j,\ell'_k)=1$ for $1\leq k,j<d$;
\smallskip
\item  $\mathrm{gcd}(\ell_j,\ell_k)=1$, $\mathrm{gcd}(\ell'_j,\ell'_k)=1, $  for $1\leq  k, j<d,\, k\neq j$;
\smallskip
\item $\ell_j\in\II_\delta,\, \mathrm{gcd}(\ell_j, z_0) =1 $;
\smallskip
\item $\ell'_j\not\in\II_\delta$.
\end{itemize}
Set
$\mathcal{F}_0=\II_{\delta}^{d-1}\cap\II_{\Delta_1}\cap\ldots\cap\II_{\Delta_{d-1}}$.
The ideal $\mathcal{F}_0$ is defined by the $(d-1)\times(d-1)$ minors of the matrix
\[
M=
\left(
\begin{array}{cccc}
z_0\ell'_1 & z_0\ell'_2 & \ldots & z_0\ell'_{d-1}\\
\ell_1 & 0 & \ldots & 0 \\
0 & \ell_2 & \ddots & \vdots \\
\vdots & \ddots & \ddots& 0\\
0 & \ldots & 0 & \ell_{d-1}
\end{array}
\right).
\]
Let $\PPT \subset\PP_1\times\PP_3$ be the blow up of $\PP_3$ at
$\delta$. Denote by $s$ (resp. $H$) the hyperplane class of $\PP_1$
(resp. $\PP_3$) and by $\widetilde{M}$ the strict transform of the entries of $M$.
 The resolution of the ideal of the proper transform of
$\cup_{i=1}^{d-1}\Delta_i$ is
\begin{equation}
  \label{eq:resolZ}
0\rightarrow\mathcal{O}^{d-1}_{\PPT}(-s)\xrightarrow{\ \widetilde{M}\ }\mathcal{O}_{\PPT}(H)\oplus\mathcal{O}^{d-1}_{\PPT}\rightarrow\II_Z((d-1)s+H)
\rightarrow 0.
\end{equation}
This surjection gives a rational map
\[
\PPT\dashrightarrow X_d\subset\mathrm{Proj}(Sym(\OO_{\PPT}^{d-1}\oplus\OO_{\PPT}(H)))\subset\PPT\times\PP_{d+2}
\]
with $\PP_{d+2}=\vert\II_Z((d-1)s+H)\vert^\vee$. Denote by $H'$ the hyperplane
class of $\PP_{d+2}$, so the class of $X_d$ in
$\mathrm{Proj}(Sym(\OO_\PPT^{d-1}\oplus\OO_\PPT(H)))$ is $(H'+s)^{d-1}$ with
the relation $H'^d=H^{'d-1}\cdot H$ (\cite[Example 8.3.4]{Fulton}). Since we have the equality
\[
(H'+s)^{d-1}\cdot H'^3=d,
\]
the image of $\PPT$ by $\vert\II_Z((d-1)s+H)\vert$ is of degree $d$ in
$\PP_{d+2}$. Moreover, from resolution~(\ref{eq:resolZ}), it is not in a
hyperplane of $\PP_{d+2}$.
\end{proof}
\begin{proof}[Proof of Proposition~\ref{pro:ruled}]
With notation of Lemma \ref{lem:ruled} choose $d-1$ points $(p_i)_{1\leq i < d}$
in general position in $X_d$. They span in $\PP_{d+2}$ a
projective space $\pi_{d-2}$ of dimension $d-2$. The projection from $\pi_{d-2}$
restricts to a birational map from $X_d$ to a $3$-dimensional projective
space. To see that the intersection $\pi_{d-2} \cap X_d$ is the reduced scheme
defined by the $(p_i)_{1\leq i < d}$, let us consider an additional general point $p_d$
of $X_d$ and denote by $\pi_{d-1}$ the projective space  spanned by the
$(p_i)_{1\leq i \leq d}$. The threefold $X_d$ is
non degenerate of degree $d$ in $\PP_{d+2}$, so $X_d\cap \pi_{d-1}$ is
the intersection of $X_d$ by a general linear space of codimension three, hence it is just the $d$ points
$(p_i)_{1\leq i \leq d}$. In conclusion $\pi_{d-2} \cap X_d$ is the reduced scheme
defined by the $(p_i)_{1\leq i < d}$. Moreover we can assume that these points
are in the locus where $\PP_3 \dashrightarrow X_d$ is an isomorphism. As a result the base scheme in $\PP_3$ of the
composition $\PP_3 \dashrightarrow X_d$ with the linear projection from $\pi_{d-2}$
is defined by the ideal $\II_{\delta}^{d-1}\cap\II_{\Delta_1}\cap\ldots\cap
\II_{\Delta_{d-1}}\cap\II_{p_1}\cap\ldots\cap\II_{p_{d-1}}$, it thus has degree $\frac{d(d-1)}{2}+(d-1)$.
\end{proof}
\begin{cor}\label{cor:dimruled}
The family $\mathcal{R}_{d,d}\subset\mathrm{Bir}_{d,d}(\PP_3,\PP_3')$ is irreducible of dimension
\[
4+3(d-1)+3(d-1)+15=6d+13.
\]
\end{cor}
\begin{pro}\label{compR44}
The closure of  $\mathcal{R}_{4,4}$ is an irreducible component of $\Bir_{4,4}(\PP_3,\PP'_3)$.
\end{pro}
\begin{proof}
Let $\mathcal{B}$ be an irreducible component of  $\Bir_{4,4}(\PP_3,\PP'_3)$
containing $\mathcal{R}_{4,4}$. Let $\phi$ be a general element of
$\mathcal{B}$, and let $\psi$ be a general element of $\mathcal{R}_{4,4}$.
According to Proposition~\ref{pro:ruled} the scheme $F^1_\psi$ has degree $9$
and $\psi$ is in $\mathcal{B}$, thus by Corollary~\ref{cor:scs} the set
$\alpha^{-1}(]-\infty,9])$ is non empty and open in $\mathcal{B}$ so the
general element $\phi$ satisfies $\deg(F^1_\phi) \leq 9$. With
Notation~\ref{notaC1C2} we have $\deg \Cdeux=12$ so $F^1_\phi$ has
a component that is not locally complete intersection.
Let $H$ be a general hyperplane of $\PP_3$, and let $\mathscr{S}$ be the
set of points $p$ such that $F^1_\phi \cap H$ is not locally complete
intersection at $p$. Denote by $Z$ (resp. $Y$) the union of the irreducible components of
$F^1_\phi \cap H$  (resp. $\Cdeux \cap H$) supported on points $p$ of
$\mathscr{S}$. We can now bound the degree of $Z$.
By Corollary~\ref{dimrelF} and Chevalley's semi-continuity theorem there is an
open subset $U$ of \linebreak$\alpha^{-1}(]-\infty, 9])$ (hence of $\mathcal{B}$) such that
$\psi\in U$ and $\pi_U\colon F_U\cdot H \to U$ is
finite (with notation of Def.~\ref{def:semi}). By Lemma~\ref{lem:scsOmeg} the set
$$
\Sigma= \{(u,p) \in F_U\cdot H \subset U\times H,~\mathrm{rank}((\Omega_{\pi_U})_{(u,p)}) \geq 2 \}
$$
is closed in $F_U\cdot H$. As a consequence $\pi_U\colon \Sigma \to U$ is also finite. Moreover it is
onto because $U\subset \alpha^{-1}(]-\infty,9])$ so any element of $U$ must
have a base scheme that is not locally complete intersection. But the base
scheme of $\psi$ have three smooth points so the degree of the fiber $\Sigma_\psi$
is at most $9-3=6$. Lemma~\ref{scsfinite} implies the inequality $\deg(\Sigma_\phi) \leq
6$ for the general element $\phi$ of $U$, hence of $\mathcal{B}$.
But $Z \subset \Sigma_{\phi}$ hence $\deg(Z) \leq 6$ and
$\deg(Y)-\deg(Z)=\deg(\Cdeux)-\deg(F^1_\phi)\geq 3$.
The cardinal of $\mathscr{S}$ is at most $2$ because $\deg(Z) \leq 6$ and
schemes of length at most $2$ are locally complete intersection.
We thus have the following possibilities:
\smallskip
\begin{itemize}
\item if $\mathscr{S}=\{p_1,p_2\}$, $p_1\neq p_2$, then $\II_Z=\II_{p_1}^2 \cap
  \II_{p_2}^2$ because $\deg(Z)\leq 6$. Therefore each component of $Y$ has
  length $4$ and $\deg(Y)-\deg(Z)=8-6$ can't be $3$.
\smallskip
\item if $\mathscr{S}=\{p\}$ with $\II_p=(x,y)$, then we have the
  following analytic classification of all non locally complete intersection
  ideals of $\CC\{x,y\}$ of colength at most $6$ (\emph{see} \cite[\S 4.2]{briancon})
\smallskip
  \begin{center}
  \begin{tabular}[b]{|c|c|p{3.8cm}|}
    \hline
    $\mathrm{length}(Z)$ &  ideal of $Z$ & length of a general complete intersection  $Y$\\ \hline
    $n+1$ & $(y^2,x y ,x^n), 2 \leq n\leq 5$ & $n+2$ \\
    $n+2$ & $(y^2+x^{n-1} , x^2 y, x^n),\,  3\leq n\leq 4 $ & $n+3$ \\
    $n+2$ & $(y^2 , x^2 y, x^n),\,  3\leq n\leq 4 $ & $2 n$ \\
    $6$ & $(y(y+x), x^2 y, x^4)$ & $7$ \\
    $6$ & $(x^3, x^2 y, x y^2, y^3)$ & $9$ \\ \hline
  \end{tabular}
\end{center}
\smallskip
and note that the only solution to $\deg(Y)-\deg(Z)\geq 3,\, \deg(Z)\leq 6$ is the last case.
\end{itemize}
Hence $F^1_\phi$ has a triple line $\Delta$ and a smooth curve of degree $3$
denoted by $\mathcal{K}$. Any bisecant line to~$\mathcal{K}$ that intersects
$\Delta$ must be in all the quartics defining $\phi$ so it is in $F^1_\phi$. As
a result $\mathcal{K}$ is a union of~$3$ disjoint lines that intersect $\Delta$,
and $\phi$ belongs to $\mathcal{R}_{4,4}$.
\end{proof}

\section{The determinantal family $\Deter$}\label{sectiondet}
\begin{nota}\label{notaLAA}
Let $\PP_1$, $\PP_3$, $\PP'_3$ be three complex projective spaces of dimension $1$, $3$,
$3$; denote by $s$, $H$, $H'$ their hyperplane class and by $L$, $A$, $A'$ the following vector spaces
\[
L=\HH^0(\OO_{\PP_1}(s)),\quad A=\HH^0(\OO_{\PP_3}(H)),\quad A'=\HH^0(\OO_{\PP'_3}(H')).
\]
\end{nota}
\subsection{Construction of $\Deter$}
\subsubsection{Description via $\PPT$}\label{constructionX}
Let $X$ be a complete intersection in $\PxPxP$ given
by the vanishing of a general section of the bundle
\[
\OO_{\PxPxP}(s+H) \oplus \OO_{\PxPxP}(s+H')
\oplus \OO_{\PxPxP}(H+H') \oplus \OO_{\PxPxP}(s+H+H')
\]
\begin{lem}\label{Xbir}
  The projections from $X$ to $\PP_3$ and $\PP'_3$
  are birational.
\end{lem}
\begin{proof}
  It is an immediate computation from the class of
\[
  X \sim (s+H)\cdot (s+H')\cdot (H+H')\cdot (s+H+H')
\]
so $X\cdot H^3=s\cdot H^3\cdot H'^3 = X\cdot H'^3.$
\end{proof}
\begin{nota}\label{notaPT3}
Let $n_0$ $($resp. $n_1)$ be a non zero section of $\OO_{\PxPxP}(s+H)$ $($resp.
$\OO_{\PxPxP}(s+H'))$ vanishing on $X$. The section $n_0$ defines in $\PP_1\times
\PP_3$ a divisor $\PPT$ isomorphic to the blow up of~$\PP_3$ along a
line $\Delta$.
\end{nota}
\begin{lem}\label{LemresolP3T}
  The complete intersection $X$ is the blow up of $\PPT$ along the curve $\Gamma$ of ideal $\IGPT$ with
  resolution defined by a general map $\widetilde{G}$
  \begin{equation}
    \label{resolP3T}
    0 \to \OO_{\PPT}(-s) \oplus  \OO_{\PPT}(-H) \oplus  \OO_{\PPT}(-s-H)
\stackrel{\widetilde{G}}{\longrightarrow}
A' \otimes \OO_{\PPT} \to \IGPT(2s+2H) \to 0.
  \end{equation}
Moreover,
\smallskip
\begin{itemize}
\item the curve $\Gamma$ is smooth irreducible of genus $5$ and  $\vert\OO_\Gamma(s)\vert$ is a
$g^1_3$,
\smallskip
\item the sheaf  $\OO_\Gamma(H)$ has degree $8$.
\end{itemize}
\end{lem}
\begin{proof}
  The resolution of $\IGPT(2s+2H)$ claimed in (\ref{resolP3T}) is just the direct image by the first
  projection of the resolution of $\OO_X(H')$ as $\OO_{\PPT\times
    \PP'_3}$-module. So the map $\widetilde{G}$ is general. Since the bundle
$A' \otimes (\OO_{\PPT}(s) \oplus  \OO_{\PPT}(H)\oplus  \OO_{\PPT}(s+H)) $ is globally generated the
curve $\Gamma$ is smooth because it is the degeneracy locus of the general map
$\widetilde{G}$ (\cite[\S~4.1]{Ban}, \emph{see} \cite[Th.~1]{Tan}). Moreover, the resolution~(\ref{resolP3T})  gives the vanishing
$\hh^1(\IGPT)=0$. As a consequence $\hh^0(\OO_\Gamma)=1$ and $\Gamma$ is
connected and smooth so irreducible.
  The dualizing sheaf of $\PPT$ is $\omega_{\PPT}=\OO_{\PPT}(-s-3H)$ so
  $\sheafExt^1(\IGPT(2s+2H),\OO_{\PPT}(s-H))=\omega_\Gamma$. Besides the functor $\sheafHom(
  \cdot , \OO_{\PPT}(s-H))$ applied to sequence $(\ref{resolP3T})$ gives the following
  exact sequence
\begin{equation}
 0\to \OO_{\PPT}(-s-3H) \to A'^\vee\otimes \OO_{\PPT}(s-H) \to
  \OO_{\PPT}(2s-H) \oplus \OO_{\PPT}(s) \oplus \OO_{\PPT}(2s)
  \to \omega_\Gamma \to 0.
\end{equation}
As a result $\hh^0(\omega_\Gamma)=5$ and $\Gamma$ has genus $5$.
  The class of $\Gamma$ is obtained from the degeneracy locus formula for
  $\widetilde{G}$ by computing the
  coefficient of $t^2$ in the serie $\frac{1}{(1-s\cdot t)(1-H\cdot
    t)(1-(s+H)\cdot t)}$. Therefore in the Chow ring of $\PPT$ one has
\[
\Gamma \sim 5s\cdot H+ 3H^2.
\]
Hence $\Gamma \cdot H$ has degree $8$, $\Gamma\cdot s$ has degree $3$ and
$\vert\OO_{\Gamma}(s)\vert$ is a $g^1_3$.
\end{proof}
\subsubsection{Description in $\PP_3$}\label{descrP3}
Let $G$ be the composition  $\widetilde{G} \circ
\left(\begin{array}[c]{ccc}
  (E_\Delta) & 0 & 0\\
  0 & 1 & 0\\
  0 & 0 & (E_\Delta)
\end{array}\right)
$ where $(E_\Delta)$ is an equation of the exceptional divisor $E_\Delta$ of $\PPT$.
\begin{pro}\label{propGP3}
 Let $\II_{\overline{\Gamma}}$ be the ideal of the projection $\overline{\Gamma}$
 of $\Gamma$ in $\PP_3$, and let $\II_\Delta$ be the
 ideal of the line $\Delta$ defined in Notation~\ref{notaPT3}. Then
\smallskip
\begin{itemize}
\item the line $\Delta$ is $5$-secant to $\overline{\Gamma}$,
\smallskip
\item one has the following exact sequence
\begin{equation}
   \label{resolGP3}
    0 \to \OO_{\PP_3}^2(-H) \oplus  \OO_{\PP_3}(-2H)
\stackrel{G}{\longrightarrow}
A' \otimes \OO_{\PP_3} \to (\II_\Delta^2 \cap\IGP)(4H) \to 0.
\end{equation}
\smallskip
\item the linear system $\vert(\II_\Delta^2 \cap\IGP)(4H)\vert$  gives a birational
map from $\PP_3$ to $\PP'_3$ that factors through~$X$.
\end{itemize}
\end{pro}
\begin{proof}
The intersection of $\Delta \cap \overline{\Gamma}$ is computed in $\PPT$ from
Lemma~\ref{LemresolP3T} by $\Gamma \cdot E_\Delta = \Gamma \cdot H - \Gamma
\cdot s =8 - 3$.
From Lemma~\ref{LemresolP3T} one gets
 the
  commutative diagram
\[
  \xymatrix{
    & 0 \ar[d] \\
0 \ar[r] &  \OO_{\PPT}^2(-H) \oplus  \OO_{\PPT}(-2H) \ar[d] \ar[r]^-{G} &
A'\otimes \OO_{\PPT}\ar@{=}[d] \ar[r] & \mathrm{coker}(G) \ar[d] \ar[r]& 0\\
0 \ar[r] & \OO_{\PPT}(-s) \oplus  \OO_{\PPT}(-H) \oplus  \OO_{\PPT}(-s-H) \ar[d]
\ar[r]^-{\widetilde{G}} &  A'\otimes \OO_{\PPT} \ar[r]&  \IGPT(2s+2H) \ar[d]\ar[r]& 0\\
& \OO_{E_\Delta}(-s) \oplus \OO_{E_\Delta}(-s-H) \ar[d]& & 0 \\
& 0
}
\]
so $\mathrm{coker}(G)$ is in the extension
$$
0 \to \OO_{E_\Delta}(-s) \oplus \OO_{E_\Delta}(-s-H) \to \mathrm{coker}(G) \to
\IGPT(2s+2H) \to 0
$$
and we have the isomorphism $\rho_*(\mathrm{coker}(G)) = \rho_*(\IGPT(2s+2H))$ where
$\rho$ is the projection from $\PPT$  to $\PP_3$. This gives the
resolution~(\ref{resolGP3}) after applying $\rho$ to the exact sequence
$$
0 \to \IGPT(2s+2H) \to \OO_\PPT(4H) \to \OO_{2E_\Delta \cup \Gamma}(4H) \to 0
$$
\end{proof}
In the next section we will prove that any trigonal curve of genus $5$ embedded
in $\PP_3$ by a general line bundle of degree $8$ is the curve $\overline{\Gamma}$ for
some choice of $X\subset\PxPxP $.
\begin{rem}
The inverse of this rational map is obtained by the same
construction from $n_1$ and~$\PP'_3$ (Notation~\ref{notaPT3}). Hence the
inverse map is also given by quartic polynomials.
\end{rem}
So let us introduce the following
\begin{definition}
 The birational maps constructed in
 Proposition~\ref{propGP3} will be called \textbf{\textit{determinantal quarto-quartic
birational map}}. They form an irreducible family denoted by $\Deter$.
\end{definition}
\begin{pro}\label{dimD44}
The family $\Deter\subset \Bir_{4,4}(\PP_3,\PP'_3)$ is  irreducible of dimension $46$.
\end{pro}
\begin{rem}
In Corollary~\ref{D44comp} we will prove that this family turns out to be an irreducible component
of $\Bir_{4,4}(\PP_3,\PP'_3)$.
\end{rem}
\begin{proof}[Proof of Proposition~\ref{dimD44}]
First let us detail the choices made to get $X$. The choice of a complete
intersection of this type is equivalent to the choices of:
\smallskip
\begin{itemize}
\item[i)] $\overline{n_0}\in \PP\left(\HH^0(\OO_\PxPxP(s+H))\right)$,
\smallskip
\item[ii)] $\overline{n_1}\in \PP\left(\HH^0(\OO_\PxPxP(s+H'))\right)$,
\smallskip
\item[iii)] $\overline{n_2}\in \PP\left(\HH^0(\OO_\PxPxP(H+H'))\right)$,
\smallskip
\item[iv)] $\overline{n_3}\in \PP\left(\HH^0(\OO_\PxPxP(s+H+H'))/(\overline{n_0}\cdot A'
  + \overline{n_1}\cdot A + \overline{n_2}\cdot L)\right)$.
\end{itemize}
\smallskip
So the choice of $X$ is made in an irreducible variety of dimension
$7+7+15+21=50$. Let $I_2\subset \PP_3\times \PP'_3$ be the divisor defined by
$\overline{n_2}$. Pushing down $\OO_X(s)$ to $I_2$ one gets the following
resolution of the ideal $\mathscr{G}_X$ of the graph of the birational
transformation constructed from $X$
\[
 0 \to L^\vee\otimes \OO_{I_2} \xrightarrow{
{\small \left(   \begin{array}[c]{c}
n_0\\ n_1\\ n_3
\end{array}\right)
}
} \OO_{I_2}(H)\oplus \OO_{I_2}(H')\oplus \OO_{I_2}(H+H')
\to \mathscr{G}_X(2H+2H') \to 0
\]
so $\dim(\Deter)=50-\dim(GL(L^\vee))=46$.
\end{proof}
In \S~\ref{explicit} we  provide an explicit construction of this map and its
inverse and give more geometric properties.
\subsection{Trigonal curves of genus $5$}
In the previous section we obtained a trigonal curve of genus $5$ naturally
embedded in $\PPT$. To understand how this construction is general, we start in
this section with an abstract trigonal curve of genus $5$, then we choose some
line bundle and explain how to obtain the previous construction. Notation of the
various line bundles introduced here will turn out to be compatible with notation
in the previous section.
\subsubsection{Model in $\planT$}\label{modelplanT}
The following result is detailed in \cite[Chap. 5]{ACGH}.
\bigskip
Let $\Gamma$ be any trigonal curve of genus $5$, let $\omega_\Gamma$ be its canonical
bundle, and denote by $\OO_\Gamma(s)$ the line bundle of degree $3$ generated by two
 sections. Without any extra choice we have the following embedding
of $\Gamma$:
\smallskip
\begin{itemize}
\item The linear system $\vert\omega_\Gamma(-s)\vert$ sends $\Gamma$ to a plane curve of
degree $5$. This plane curve has a unique singular point $p_0$ (a node or an
ordinary cusp).
\smallskip
\item Let $\planT$ be the blow up of this plane at $p_0$, denote by
$E_0$ the exceptional divisor and by $h$ the hyperplane class of this plane:
\[
\vert\OO_{\planT}(h)\vert\colon \planT \longrightarrow \PP_2=\vert\omega_\Gamma(-s)\vert^\vee
\]
\end{itemize}
\subsubsection{Embeddings of degree $8$ in $\PP_3$}\label{Sec:deg8}
With notation of section~\ref{modelplanT}. Let $\OO_\Gamma(H)$ be
a general line bundle of degree $8$ on $\Gamma$.
 One has $\hh^0(\OO_\Gamma(H-s))=1$. The vanishing locus of this section defines
 in $\Gamma$ a unique effective divisor $D'_5$ of degree $5$ such that $\OO_\Gamma(H)=\OO_\Gamma(D'_5+s)$.
Moreover $\hh^0(\OO_\Gamma(2h-D'_5))$ is also $1$ so there is a
 unique effective divisor $D_5$ on $\Gamma$ such that $\OO_\Gamma(2h)=\OO_\Gamma(D_5+D'_5)$. Note
 that both $D_5$ and $D'_5$ have degree $5$.
 So the picture in $\PP_2$ looks like this. The curve $\Gamma$ is sent via
 $|\OO_\Gamma(h)|$ to a quintic curve singular at $p_0$. This curve contains two
 divisors $D_5, D'_5$ of degree $5$ such that $D_5+D'_5$ is the complete
 intersection of the plane quintic with a conic. From the general assumptions on
 $H$ (hence on $D'_5$) the conic is smooth, it doesn't contain the point $p_0$
 and both $D_5$  and $D'_5$ are supported by five distinct points.
Let $p_1$, $p_2$, $\ldots$, $p_5$ be the five points of $D_5$, and let $S_3$ be the blow up of
$\PP_2$ at the six points $p_0$, $p_1$, $\ldots$, $p_5$. Let $(E_i)_{0\leq i \leq 5}$ be the
corresponding exceptional divisors. Remark that $\OO_\Gamma(H)$ is the restriction to
$\Gamma$ of $\OO_{S_3}(3h-\sum_{i=0}^5 E_i)$;  the linear system
$\vert\OO_\Gamma(H)\vert$ embeds $\Gamma$ on the cubic surface $S_3$ in $\PP_3$. Denote by
$\overline{\Gamma}$  this curve of degree $8$ in $\PP_3$ and keep the notation $H$
for the class $3h-\sum_{i=0}^5 E_i$ on $S_3$.
\medskip
We can summarize these data in the following diagram:
\[
  \xymatrix{
  &  \Gamma \ar[ld]_-{\vert\OO_\Gamma(s)\vert}\ar[d]\ar[r]^-{\vert\OO_\Gamma(H)\vert}_-{\sim}& \overline{\Gamma} \ar@{}[r]|-*[@]{\subset} &
   S_3 \ar@{=}[d] \ar@{}[r]|-*[@]{\subset} & \PP_3 \\
   \PP_1  & \planT=\widetilde{\PP_2(p_0)} \ar[l]\ar[d] & &\widetilde{\PP_2(p_0,p_1,\ldots,p_5)} \ar[ll] \\
    & \vert\OO_\Gamma(h)\vert^\vee = \PP_2
  }
\]
\begin{rem}
Any irreducible smooth curve of degree $8$ and genus $5$ in $\PP_3$ with a
$5$-secant line $\Delta$ is trigonal.
\end{rem}
\begin{proof}
Indeed, the planes  containing $\Delta$ give a base point free linear system of degree $3$.
\end{proof}
\begin{nota}
Let $\Delta$ be the line of $S_3$ equivalent to $2h-\sum_{i=1}^5 E_i$. One has
$\Delta \cap \overline{\Gamma}=D'_5$ and $\Delta$ is the unique line $5$-secant
to $\overline{\Gamma}$.
Denote by $\PPT$ the blow up of $\PP_3$
in $\Delta$ and by $\IGPT$ the ideal of the proper transform of
$\overline{\Gamma}$.
\end{nota}
Keeping notation $s$, $H$ for the pull back of the hyperplane classes of
$\PP_1$ and $\PP_3$ on $\PPT$ one gets the following
\begin{lem}\label{suiteconormale}
  We have an exact sequence
\[
0 \longrightarrow \OO_{\PPT}(s) \longrightarrow \IGPT(2s+2H) \longrightarrow
\OO_{S_3}(H-E_0) \longrightarrow 0.
\]
  Hence $\hh^0(\IGPT(2s+2H))=4$ and the sheaf $\IGPT(2s+2H)$ is globally
  generated.
\end{lem}
\begin{proof}
As $\Delta\subset S_3$ the ideal of $S_3$ in $\PPT$ is $\II_{S_3 \mid
  \PPT}=\OO_{\PPT}(-s-2H)$. The class of $\overline{\Gamma}$ in $S_3$ is
$5h-2E_0-\sum_{i=1}^5 E_i$ and $s\sim H-\Delta \sim h-E_0$; the ideal of
$\overline{\Gamma}$ in $S_3$ is thus $\II_{\Gamma \mid S_3} = \OO_{S_3}(-H-2s-E_0)
$. Hence the following exact sequence of ideals
\[
0 \longrightarrow \II_{S_3 \mid \PPT} \longrightarrow \IGPT \longrightarrow \II_{\Gamma \mid S_3} \longrightarrow 0
\]
twisted by $2s+2H$ gives the exact sequence  of the statement. To conclude just remark that
$\hh^1(\OO_{\PPT}(s))=0$ and that both $\OO_{\PPT}(s)$ and $\OO_{S_3}(H-E_0)$ have
two sections and are globally ge\-nerated.
\end{proof}
\begin{pro}\label{ProkerN}
 Let $A'$ be the $4$-dimensional vector space $\HH^0(\IGPT(2s+2H))$; let us
 consider~$\mathscr{N}$ defined by the following exact sequence
\[
0 \longrightarrow \mathscr{N} \longrightarrow A'\otimes \OO_{\PPT}
\longrightarrow \IGPT(2s+2H) \longrightarrow 0.
\]
We have
\[
\mathscr{N} \simeq \OO_{\PPT}(-s) \oplus  \OO_{\PPT}(-H) \oplus  \OO_{\PPT}(-s-H).
\]
\end{pro}
\begin{proof}
First remark that $\mathscr{N}$ is locally free. Indeed, the sheaf $\IGPT$ is
the ideal of a $1$-dimensional scheme of $\PPT$ without embedded nor isolated
points. As a consequence for all closed point $x$ of $\PPT$ we have the vanishing of the
following localizations
\[
\forall\, i\geq 2,\quad \sheafExt ^i(\II_{\Gamma \mid \PPT, x}, \OO_{\PPT,x} ) =0
\]
so
\[
\forall\, i\geq 1,\quad \sheafExt ^i(\mathscr{N}_x, \OO_{\PPT,x} ) =0
\]
and $\mathscr{N}$ is locally free.
Now construct an injection from $ \OO_{\PPT}(-s) \oplus  \OO_{\PPT}(-H) \oplus
\OO_{\PPT}(-s-H)$ to $\mathscr{N}$.
\smallskip
\begin{itemize}
\item Since
  $\hh^0(\IGPT(3s+2H))=\hh^0(\OO_{\PPT}(3s)+\hh^0(\OO_{S_3}(2H-E_0-\Delta))=7$ and
  $\hh^0(A'\otimes \OO_{\PPT}(s))=8$ we have $\hh^0(\mathscr{N}(s))\geq 1$.
\smallskip
\item $\hh^0(\mathscr{N}(H))\geq 2$ because
  $\hh^0(\IGPT(2s+3H))=\hh^0(\OO_{\PPT}(s+H))+\hh^0(\OO_{S_3}(2H-E_0))=14$ and
  $\hh^0(A'\otimes \OO_{\PPT}(H))=16$.
\smallskip
\item As
  $\hh^0(\IGPT(3s+3H))=\hh^0(\OO_{\PPT}(2s+H))+\hh^0(\OO_{S_3}(3H-E_0-\Delta))=21$
  and $\hh^0(A'\otimes \OO_{\PPT}(s+H))=28$ we have
  $\hh^0(\mathscr{N}(s+H))\geq 7$.
\end{itemize}
\smallskip
We are thus able to find a non zero section $\sigma_0$ of $\mathscr{N}(s)$ then a
section $\sigma_1$ of $\mathscr{N}(H)$ independent with $E_\Delta \cdot
\sigma_0$ (with notation of \S~\ref{descrP3}), then a section $\sigma_2$ of
$\mathscr{N}(s+H)$ not in $A\otimes \sigma_0 \oplus L\otimes \sigma_1$ (Notation~\ref{notaLAA}).
As a result $\sigma_0$, $\sigma_1$ and $\sigma_2$ give an injection of $
\OO_{\PPT}(-s) \oplus  \OO_{\PPT}(-H) \oplus
\OO_{\PPT}(-s-H)$ into $\mathscr{N}$. Moreover these two vector bundles have the same first
Chern class, so it is an isomorphism.
\end{proof}
In other words, Proposition \ref{ProkerN} gives the existence of the map
$\widetilde{G}$ from Lemma~\ref{LemresolP3T}
for any trigonal curve $\Gamma$ of genus $5$, therefore we have
\begin{cor}\label{cor:descrd44}
Let $\overline{\Gamma}$ be a trigonal curve of genus $5$ embedded in $\PP_3$ by
a general linear system of degree $8$. The quartics containing
$\overline{\Gamma}$ and singular along its unique $5$-secant line give an element
of~$\Deter$.
\end{cor}
\begin{cor}\label{D44comp}
The closure of $\Deter$ is an irreducible component of $\Bir_{4,4}(\PP_3,\PP'_3)$.
\end{cor}
\begin{proof}
From Corollary~\ref{dimD44} the family $\Deter$ is already irreducible. Choose an
irreducible component $\mathcal{B}$ of $\Bir_{4,4}(\PP_3,\PP'_3)$ containing
$\Deter$. Let $\phi$ be a general element of $\mathcal{B}$, and let $\psi$ be a general
element of $\Deter$.
With notation of \S\ref{NOTATIONS} the degree of the $1$-dimensional
scheme $F_\psi$ is $11$ and $\psi$ is in $\mathcal{B}$. So by
Corollary~\ref{cor:scs} for the family $\mathcal{B}$, the degree of the $1$-dimensional
scheme $F_\phi$ is at most $11$ because $\phi$ is general in $\mathcal{B}$. For
a general hyperplane $H'$ of $\PP'_3$ the quartic surface $\phi^{-1}(H')$ is
thus not normal. Indeed, we proved in Lemma~\ref{lem:normalrational} that if
$\phi^{-1}(H')$ is normal then the
schemes $F^1_\phi$ and $\Cdeux$ are equal and $\deg F_\phi=12$.
With Definition~\ref{def:semi} we have $\dim(\SingF_\phi)=1$, hence
$\deg(\SingF_\phi)\geq 1$. But with notation of Corollary~\ref{cor:scs} we have
$\beta(\phi)=1, \beta(\psi)=1$ and
$\eta(\psi)=1$. As a consequence by semi-continuity of $\eta$ we have $\deg(\SingF_\phi) = 1$ for the general element $\phi$ of $\mathcal{B}$.
There is thus a line $\Delta$ in $\PP_3$ such that $\phi^{-1}(H')$ is singular
on $\Delta$. Denote again by $\PPT$ the blow up of $\PP_3$ in $\Delta$, then $\phi$ factors through a
rational map $\widetilde{\phi}\colon\PPT \dashrightarrow \PP'_3$ via a linear
subsystem $W_\phi$ of $\vert\OO_\PPT(2s+2H)\vert$.
Now adapt Notation~\ref{notaC1C2} to geometric liaison in $\PPT$ instead of
$\PP_3$. Let $\ell'$ be a general line
of $\PP'_3$, set $\widetilde{C}_1=\widetilde{\phi}^{-1}_*(\ell')$, and let
$\widetilde{C}_2$ be the residual of $\widetilde{C}_1$ in
the complete intersection $\widetilde{\phi}^{-1}(\ell')$. The curves
$\widetilde{C}_1$ and $\widetilde{C}_2$ are geometrically linked by
two hypersurfaces of class $2s+2H$. The liaison exact
sequence gives
\[
0 \longrightarrow \II_{\widetilde{C}_1\cup \widetilde{C}_2}(2s+2H) \longrightarrow \II_{\widetilde{C}_2}(2s+2H) \longrightarrow \omega_{\widetilde{C}_1}(H-s) \longrightarrow 0.
\]
By Lemma~\ref{lem:scsOmeg} the curves $C_1$, $\widetilde{C}_1$ and $\widetilde{C}_2$ are smooth
because it is already true for $\widetilde{\psi}$. Thus $\widetilde{C}_2$ is
the base scheme of $\widetilde{\phi}$ and all the elements of $W_\phi$ vanish
on $\widetilde{C}_2$. Hence $\hh^0(\II_{\widetilde{C}_2}(2s+2H))\geq 4$ so we
have $\hh^0(\omega_{\widetilde{C}_1}(H-s))\geq 2$ on the
smooth rational curve $\widetilde{C}_1$. Hence $\OO_{\widetilde{C}_1}(H-s)$ has
degree at least $3$. The class of $\widetilde{C}_1$ in the Chow ring of $\PPT$
is thus $\widetilde{C}_1 \sim (4-a).H^2+a H\cdot s$ with $a\geq 3$. But
$\widetilde{C}_1 \cdot s = 4-a \geq 1$ because $C_1$ is irreducible of degree
$4$ and not in a plane containing $\Delta$. So $a=3$ and the class of $\widetilde{C}_2$ is $  3H^2+5 H\cdot s$. As a consequence $\vert\OO_{\widetilde{C}_2}(s)\vert$ is a $g^1_3$.
The restrictions to $\widetilde{C}_1$ and $\widetilde{C}_2$ of the exact sequences
\[
0 \longrightarrow \omega_{\widetilde{C}_{3-i}} \longrightarrow  \omega_{\widetilde{C}_1\cup \widetilde{C}_2} \longrightarrow \omega_{\widetilde{C}_1\cup \widetilde{C}_2} \otimes
\OO_{\widetilde{C}_i} \longrightarrow 0 \qquad i\in\{1,2\}
\]
show that  $\widetilde{C}_1\cap \widetilde{C}_2$ has length $9$ and that $\widetilde{C}_2$ has genus $5$. As a result the
base locus of $\widetilde{\phi}$ is a trigonal curve of genus $5$, so $\phi$ belongs to $\Deter$ by Corollary \ref{cor:descrd44}.
\end{proof}
\subsection{Explicit construction of the birational map and its inverse}\label{explicit}
In the next paragraphs we want to expose properties similar to the usual properties of the
classical cubo-cubic transformation such as explicit computation of the inverse
and isomorphism between the base loci (\cite[\S 7.2.2]{CAG}).
\subsubsection{Explicit construction of $\overline{\tau}\colon \PP_3 \dashrightarrow  \PP'_3$}
In the proof of Proposition~\ref{dimD44} we detailed the choices needed to
construct the complete intersection $X\subset \PxPxP$. Let us replace them
by the choice of some linear maps.
\begin{definition}
With Notation~\ref{notaLAA},  choose general linear maps
$N_0$, $N_1$, $M$, $T$ between the following vector spaces
\[
N_0\colon  L^\vee \longrightarrow A,\quad N_1\colon  L^\vee \longrightarrow A',\quad M\colon
A^\vee \longrightarrow A', \quad
T\colon\begin{array}[t]{ccc}
  L^\vee & \longrightarrow & \Homdroit(A^\vee, A')\\
  \lambda      & \longmapsto & T_\lambda
\end{array}.
\]
For any $\lambda \in L^\vee$ the image of $\lambda$ by $T$ will be
denoted by $T_\lambda$.
Choose also generators for $\wedge^2 L, \wedge^4 A, \wedge^4 A'$ and denote the
corresponding isomorphisms by
\[
B\colon L\stackrel{\sim}{\longrightarrow} L^\vee,\quad \alpha\colon\wedge^3A \stackrel{\sim}{\longrightarrow}
A^\vee,\quad \alpha'\colon \wedge^3A' \stackrel{\sim}{\longrightarrow}  A'^\vee.
\]
The duality bracket between a vector space and its dual will be denoted by
$<\cdot, \cdot >$.
\end{definition}
For any $z$ in $A^\vee$ let us define the elements $g_0$, $g_1$, $g_2$ of $A'$ by
\[
g_0=N_1\circ B \circ \transp{N_0}(z),\quad g_1=M(z),\quad g_2=T_{B\circ \transp{N_0}(z)}(z).
\]
Now consider the quartic map
\[
\tau\colon
\begin{array}[t]{lll}
  A^\vee & \longrightarrow & A'^\vee \\
  z & \longmapsto & \alpha'(g_0\wedge g_1\wedge g_2)
\end{array}
\]
and also the rational map $\overline{\tau}\colon\PP_3 \dashrightarrow \PP'_3$ induced by $\tau$.
\begin{example}
Take basis for $L$, $A$, $A'$ and choose the identity matrix for $M$. With
\[
N_0=N_1=\left(\begin{array}{cc}
1 & 0 \\
0 & 1 \\
0 & 0 \\
0 & 0
\end{array}\right),\,
U_0=\left(\begin{array}{cccc}
0 & 1 & 0 & 0 \\
0 & 0 & 1 & 0 \\
0 & 0 & 0 & 1 \\
1 & 0 & 0 & 0
\end{array}\right),\,
U_1=\left(\begin{array}{cccc}
0 & 0 & 0 & 1 \\
1 & 0 & 0 & 0 \\
0 & 1 & 0 & 0 \\
0 & 0 & 1 & 0
\end{array}\right),
\]
define  $T\colon(\lambda_0,\lambda_1)\mapsto\lambda_0 U_0+\lambda_1 U_1$. Then we have
\[
G=
\left(\begin{array}{ccc}
-z_{1} & z_{0} & -z_{1}^{2}+z_{0} z_{3} \\
z_{0} & z_{1} & z_{0}^{2}-z_{1} z_{2} \\
0 & z_{2} & z_{0} z_{1}-z_{1} z_{3} \\
0 & z_{3} & -z_{0} z_{1}+z_{0} z_{2}
\end{array}\right).
\]
Denote by $\Delta_k(G)$ the determinant of the matrix $G$ with its $k$-th line removed.
The map $\overline{\tau}$ defined by
\[
(z_0:z_1:z_2:z_3) \dashrightarrow (\Delta_1(G): -\Delta_2(G) : \Delta_3(G) :
-\Delta_4(G))
\]
is birational (see \S~\ref{explicittmoins1}) and is in the closure of $\Deter$.
\end{example}

\subsubsection{Explicit construction of $\overline{\tau}^{-1}$}\label{explicittmoins1}
For any $y$ in $A'^\vee$ denote by $g'_0$, $g'_1$, $g'_2$ the following elements of $A$
\[
g'_0=-N_0\circ B \circ \transp{N_1}(y),\quad g'_1=\transp{M}(y),\quad g'_2=\transp{(T_{B\circ \transp{N_1}(y)})}(y).
\]
Now consider the quartic map
\[
\tau'\colon
\begin{array}[t]{lll}
  A'^\vee & \longrightarrow & A^\vee \\
  y & \longmapsto & \alpha(g'_0\wedge g'_1\wedge g'_2)
\end{array}
\]
and also the rational map $\overline{\tau'}\colon\PP'_3 \dashrightarrow \PP_3$ induced by $\tau'$.
\begin{lem}\label{lemrel}
  For $z\in A^\vee$ and $y\in A'^\vee$ we have
\[
y\wedge\tau(z)=0 \iff \forall\, i \in\{0,1,2\},\, <y,g_i>=0 \iff z\wedge\tau'(y)=0.
\]
\end{lem}
\begin{proof}
By definition of $\alpha$ and $\alpha'$ we have
\[
  \begin{array}[c]{ccc}
  y\wedge\tau(z)=0 & \iff & \forall\, i \in\{0,1,2\} <y,g_i>=0,\\
z\wedge\tau'(y)=0 & \iff & \forall\, i \in\{0,1,2\} <z,g'_i>=0.
  \end{array}
\]
Now remark that $<y,g_0>=0 \iff <\transp{N_1}(y),B\circ\transp{N_0}(z)>=0\iff
 <g'_0,z>=0$. Note that it is also  equivalent to the proportionality of
 $\transp{N_1}(y)$ and $\transp{N_0}(z)$. Furthermore remind that $T$ is  general in
$\Homdroit(L^\vee,\Homdroit(A^\vee,A'))$, so
it is injective. Thus we have $<y,g_0>=0$ if and only if $T_{B\circ\transp{N_1}(y)}$ and
$T_{B\circ\transp{N_0}(z)}$ are proportional in $\Homdroit(A^\vee,A')$.
Therefore we have
\[
\left\{\begin{array}[c]{c}
  <y,g_0>=0 \\
  <y,g_2>=0
\end{array}\right. \iff
\left\{
\begin{array}[c]{c}
T_{B\circ\transp{N_1}(y)} \wedge T_{B\circ\transp{N_0}(z)} =0\\
<y , T_{B\circ\transp{N_1}(y)}(z) > =0
\end{array}\right.
\iff
\left\{
\begin{array}[c]{c}
<g'_0,z>=0 \\
<g'_2,z>=0
\end{array}\right.
.
\]
We achieve the proof because $<y,g_1>=<g'_1,z>$.
\end{proof}
\begin{cor} For $(\lambda,z)\in L^\vee \times A^\vee$, consider the
 elements of $A'^\vee$ given by
\[
\widetilde{g_0}=N_1(\lambda),\ \widetilde{g_1}=M(z),\ \widetilde{g_2}=T_\lambda(z).
\]
 Thus $\overline{\tau'}\circ \overline{\tau} = \mathrm{id}_{\PP_3}$ and
 $\overline{\tau}$ is the element of $\Deter$ constructed with
\[
\widetilde{G}_{(\lambda,z)}=\left(\widetilde{g_0}, \widetilde{g_1}, \widetilde{g_2}
 \right) \mbox{ and } G_z=\left( g_0,g_1,g_2 \right).
\]
\end{cor}
\begin{proof}
Lemma \ref{lemrel} directly implies  that $\overline{\tau'}\circ
\overline{\tau}$  is the identity map of $\PP_3$. Remind that for $(\lambda,z) \in L^\vee \times
A^\vee$ the corresponding point of $\PP_1\times \PP_3$ is in $\PPT$ if and only
if $\mathop{<\lambda, \transp{N_0}(z)>}=0$. But it is equivalent to the proportionality of $\lambda$
and $B\circ \transp{N_0}(z)$. So for all $y\in A'^\vee$ and
$(\overline{\lambda},\overline{z})\in \PPT$, we have  $y\wedge\tau(z)=0$ if and only if
$<y,\widetilde{g}_i>=0$ for any $i\in\{0,1,2\}$. Thus $\overline{\tau}$ is the element of
$\Deter$ constructed from $\widetilde{G}$.
\end{proof}
\subsubsection{The isomorphism between $\Gamma$ and $\Gamma'$}
\begin{pro}\label{propHH}
Let  $\overline{\tau}\colon \PP_3 \dashrightarrow \PP'_3 $ be a element of
$\Deter$. Let  $\Gamma$ $($resp. $\Gamma')$ be the trigonal curve of genus $5$ defining
$\overline{\tau}$ $($resp. $\overline{\tau}^{-1})$. Then $\Gamma'$ is isomorphic to
$\Gamma$ and the embedding of $\Gamma$ in $\PP'_3$ is given by
\[
\OO_{\Gamma}(H') = \omega_{\Gamma}^{\otimes 2}(-H).
\]
\end{pro}
We will prove this result after the construction in Lemma~\ref{QQP2T} of two
exact sequences over~$\planT$.
\begin{nota}
In the context of \S~\ref{modelplanT} we have
\[
\planT \subset \PP_1 \times \PP_2,\ U=\HH^0\left(\OO_{\PP_2}(h)\right),\
L=\HH^0\left(\OO_{\PP_1}(s)\right).
\]
The equation of $\planT $ in $ \PP_1 \times \PP_2$ gives a linear map
\[
N_2\colon L^\vee \longrightarrow U.
\]
Let $\OO_{\planT}(-s) \stackrel{\Lambda}{\longrightarrow} L^\vee \otimes \OO_{\planT} $
be the pull back of the tautological injection on $\PP_1$, and let $Q$, $Q'$ be the
cokernels of $N_0 \circ \Lambda$ and  $N_1 \circ \Lambda$. Hence we have the exact
sequences defining $\pi$ and $\pi'$
\[
0 \to \OO_{\planT}(-s)    \xrightarrow{ N_0 \circ \Lambda }  A \otimes
\OO_{\planT} \stackrel{\pi}{\longrightarrow} Q \to 0,
\]
\[
0 \to \OO_{\planT}(-s)    \xrightarrow{ N_1 \circ \Lambda }  A' \otimes
\OO_{\planT} \xrightarrow{\pi'} Q' \to 0.
\]
\end{nota}
\begin{rem}
Both $Q$ and $Q'$ are isomorphic to $\OO_{\planT}^2 \oplus \OO_{\planT}(s)$, but
we prefer to not choose one of these isomorphisms to avoid confusions.
\end{rem}
\begin{itemize}
\item Choose a linear map $\gamma\colon U^\vee \rightarrow \CC$ such that the map
\[
\begin{array}[t]{lll}
     U^\vee & \to & \CC \oplus L \\
     u & \mapsto & (\gamma(u), \transp{N_2}(u))
  \end{array}
\]
has rank $3$.
\smallskip
\item Define
\[
\widetilde{R}\colon
  \begin{array}[t]{lll}
    U^\vee& \longrightarrow & \Homdroit(A^\vee,A') \\
    u & \longmapsto & \widetilde{R}_u=\gamma(u)\cdot M + T_{B\circ \transp{N_2}(u)}
  \end{array}
\]
and denote by $\widetilde{R}_u$ the image of $u$ by $\widetilde{R}$.
\item Set $R=\pi' \circ \widetilde{R} \circ \transp{\pi}$.
\end{itemize}
We have the
  commutative diagram
\[
\xymatrix{
  A^\vee(-h)  \ar[r]^{\widetilde{R}} & A'\otimes \OO_{\planT}\ar@{->>}[d]^{\pi'}\\
Q^\vee(-h) \ar@{^{(}->}[u]^{\transp{\pi}} \ar[r]^{R}
&Q'}
\]
\begin{lem}\label{QQP2T}
  We have the following exact sequences of $\OO_{\planT}$-modules
\[
0 \longrightarrow Q^\vee(-h) \xrightarrow{R} Q' \longrightarrow \OO_{\Gamma}(H') \longrightarrow 0,
\]
\[
0 \longrightarrow Q'^\vee(-h) \xrightarrow{\transp{R}} Q \longrightarrow \OO_{\Gamma}(H) \longrightarrow 0.
\]
\end{lem}
\begin{proof}
Remark first that the class in $\mathrm{Pic}(\planT)$ of the determinant of $R$ is $2s+3h$ and
that it is also the class of $\Gamma$ in $\planT$. Nevertheless  we need to explicitly relate
points $p$ of $\planT$ such that $R_p$ is not injective to points $p$ of $\PPT$ such that
$\widetilde{G}_p$ is not injective.
So consider $(\lambda,u)$ in $L^\vee \times U^\vee$ such that the corresponding point
$(\overline{\lambda}, \overline{u})$ of $\PP_1\times \PP_2$ is in $\planT$. In other words we have
\[
<\lambda, \transp{N_2}(u) > = 0.
\]
The map $R_{(\overline{\lambda}, \overline{u})}$  is not injective if and only if
there exists $z$ in the image of $\transp{\pi}_{(\overline{\lambda}, \overline{u})} $ such
that $\widetilde{R}_{u}(z)$ and $N_1(\lambda)$ are proportional. Therefore
\[
(R_{(\overline{\lambda}, \overline{u})} \mbox{ not injective  })
\iff
(
\exists\, z\in A^\vee-\{0\},\ <z,N_0(\lambda)>=0 \mbox{ and } \widetilde{R}_u(z)\wedge N_1(\lambda)=0
).
\]
By definition  $\widetilde{R}_u(z)=\gamma(u)\cdot M(z) + T_\lambda(z)$ so
$\widetilde{R}_u(z)\wedge N_1(\lambda)=0$ implies that $N_1(\lambda),M(z),T_\lambda(z)$ are linearly
dependent elements of $A'$. Now remind that $<z,N_0(\lambda)>=0$ is just the
equation of $\PPT$ in $\PP_1\times \PP_3$, we thus have
\[
\left(
  (\overline{\lambda},\overline{u}) \in \planT   \mbox{ and }
R_{(\overline{\lambda}, \overline{u})}\mbox{ not injective }
\right)
\Rightarrow
\left(
  \exists\, \overline{z}\in \PP_3 \mbox{ such that }(\overline{\lambda}, \overline{z})\in\PPT \mbox{ and
  }\widetilde{G}_{(\overline{\lambda}, \overline{z})}\mbox{ not injective }
\right).
\]
Conversely, if $(\overline{\lambda}, \overline{z})\in\PPT$  and   $N_1(\lambda),M(z),T_\lambda(z)$
are linearly dependent then there exists by definition of $\gamma$ an element $u\in U^\vee-\{0\}$ such that
$(\gamma(u)\cdot M(z) + T_\lambda(z))\wedge N_1(\lambda)=0$ and
$<\lambda,\transp{N_2}(u)>=0$. As a result
\[\left(
  (\overline{\lambda}, \overline{z})\in\PPT\mbox{ and }
  \widetilde{G}_{(\overline{\lambda}, \overline{z})}\mbox{ not injective }
\right)
\Rightarrow
\left(\exists\, \overline{u}\in \PP_2
\mbox{ such that }(\overline{\lambda}, \overline{u})\in\planT \mbox{ and
  }R_{(\overline{\lambda}, \overline{u})}\mbox{ not injective }
\right).
\]
These two implications show that the cokernel of $\transp{R}$ is
$\OO_\Gamma(H)$. The statement follows by exchanging $\PP_3$ and $\PP'_3$.
\end{proof}
\begin{proof}[Proof of Proposition~\ref{propHH}]
  Lemma \ref{QQP2T} implies
\[
  \OO_\Gamma(H')=\sheafExt^1(\OO_\Gamma(H),\OO_\planT(-h))=\omega_\Gamma\otimes
  \OO_\planT(s+h)=\omega_\Gamma^{\otimes 2}(-H).
\]
\end{proof}
\subsection{Geometric description of the $P$-locus}\label{subsec:Plocus}
\subsubsection{The ruled surface $S_9$}
Remind from \S~\ref{constructionX} that  $\overline{\tau}'\colon \PP'_3 \dashrightarrow \PP_3$ factors through $X
\to \PPT$. Therefore if $N_\Gamma$ denotes the normal bundle of $\Gamma$ in $\PPT$,
Sequence~(\ref{resolP3T}) shows that $\overline{\tau}'$ contracts the ruled
surface $\widetilde{S}_9'=\mathrm{Proj}(Sym(N_\Gamma^\vee(2s+2H)))$ to $\Gamma$. Moreover,
the exact sequence of Lemma~\ref{suiteconormale} restricted to $\Gamma$ gives
\[
0 \longrightarrow \OO_\Gamma(s) \longrightarrow N_\Gamma^\vee(2s+2H)
\longrightarrow  \OO_\Gamma(H-E_0) \longrightarrow 0 .
\]
Sections of $\OO_\Gamma(s)$ are thus hyperplane sections of
$\mathrm{Proj}(Sym(N_\Gamma^\vee(2s+2H)))$ and each of them contains three
rules of this surface. So the projection of this surface in $\PP'_3$ is the
union of the "triangles" with
vertices given by the $g^1_3$.
In a symmetric way, the $P$-locus of $\overline{\tau}$ contains a ruled
surface $S_9\subset \PP_3$. The basis of $S_9$ is isomorphic to $\Gamma$ and its
rules are the lines of $\PP_3$ that
cut an element of the $g^1_3$ in length $2$. In other words, $S_9$ is the union
of the bi-secant lines of the intersection of $\overline{\Gamma}\subset
\PP_3$ with any plane containing the $5$-secant $\Delta$. The smooth model
$\widetilde{S}_9$ of $S_9$ is an extension of $\OO_\Gamma(H'-E_0)$ by
$\OO_\Gamma(s)$; hence $S_9$ has
degree $9$ and $\overline{\tau}(S_9)$ is the curve $\Gamma$ embedded in $\PP'_3$
by $\OO_\Gamma(H')=\omega_\Gamma^{\otimes 2}(-H)$ (Proposition~\ref{propHH}).
\bigskip
It is easy to explain geometrically why $S_9$ is contracted by
$\overline{\tau}$. Indeed, by the previous description, the rules of $\Gamma$
intersect the scheme defined by the ideal $\II_\Delta^2 \cap
\II_{\overline{\Gamma}}$ in length $4$, so these lines are contracted by $\vert\II_\Delta^2 \cap
\II_{\overline{\Gamma}}(4H)\vert$.
\subsubsection{The cubic surface $S_3$}
It is interesting to understand geometrically why the cubic surface $S_3$
defined in \S~\ref{Sec:deg8} is also contracted by
$\overline{\tau}$ because it can't be explained anymore by $4$-secant lines.
This time, consider the planes of $\PP_3$ containing the line $E_0\subset
S_3$. The general element of the linear system $\vert\OO_{S_3}(H-E_0)\vert$ gives a
conic of $\PP_3$  intersecting $\Delta$ in a point and $\overline{\Gamma}$ in
six points. Any of these conics thus cuts the scheme defined by $\II_\Delta^2 \cap
\II_{\overline{\Gamma}}$ in length at least $8$. So these conics are contracted
by $\vert\II_\Delta^2 \cap \II_{\overline{\Gamma}}(4H)\vert$ and
$\overline{\tau}(S_3)$ is the line $5$-secant to the image of $\Gamma$ in $\PP'_3$.
\bigskip
We have achieved the description of the $P$-locus of $\overline{\tau}$ with
the following
\begin{cor}
The jacobian of $\overline{\tau}$ is the product of equations of $S_9$ and $S_3$.
\end{cor}
\begin{proof}
Indeed, the jacobian of $\overline{\tau}$ has degree $12$ and vanishes on $S_9$
and $S_3$.
\end{proof}
\section{On the components of $\Bir_{4,4}(\PP_3,\PP'_3)$}\label{sec:othercomp}
In the previous sections we found elements $\phi$ of $\Bir_{4,4}(\PP_3,\PP'_3)$
such that for a general hyperplane~$H'$ of $\PP'_3$, the quartic surface
$\phi^{-1}(H')$ has the following properties:
\smallskip
\begin{itemize}
\item $\phi \in \mathcal{J}_{4,4}$ $\iff$ $\phi^{-1}(H')$ is normal
  (Lemma~\ref{lem:normalrational}) $\iff$ the genus of $\phi$ is $3$.
\smallskip
\item If $\phi$ belongs to $\Deter$, then $\phi^{-1}(H')$ has
  a double line; furthermore $\phi^{-1}(H')$ has
  a double line if and only if the genus of $\phi$ is $2$.
\smallskip
\item The singular locus of $\phi^{-1}(H')$ cannot contain two disjoint lines or
  a reduced cubic curve (Corollary~\ref{majlieusing}).
\smallskip
\item If $\phi$ is in $\mathcal{R}_{4,4}$, then $\phi^{-1}(H')$ has
  a triple line and the genus of $\phi$ is $0$.
\end{itemize}
\smallskip
So it remains to provide a genus one example or to know if $\phi^{-1}(H')$ can be
singular along a conic. This is achieved with the following example.
\subsection{The family $\mathcal{C}_{4,4}$}\label{familleC44}
\begin{pro}\label{pro:c44}
Let $p$ be a point of $\PP_3$. Choose some coordinates such that the ideal of $p$
in $\PP_3$ is $\II_p=(z_0,z_1,z_2)$. Take $Q_1$, $Q_2$, $f$ general in the following
spaces
\[
Q_1\in\HH^0(\II_p(2H)),\, Q_2\in\HH^0(\II_p^2(2H)),\, f\in \HH^0(\OO_{\PP_3}(H))
\]
and a general point $p_1$ of $\PP_3$. Then
\smallskip
\begin{itemize}
\item The ideal
\[
\II=(f,Q_1)^2\cap(Q_1,Q_2)\cap\II_p^2\cap\II_{p_1}
\]
gives a $3$-dimensional linear system $\vert\II(4H)\vert$.
\smallskip
\item Any isomorphism $\vert\II(4H)\vert^\vee \simeq \PP'_3$  defines a $(4,4)$ birational
map $\phi$ from $\PP_3$ into $\PP'_3$ such that $\phi^{-1}(H')$ is
a quartic with a double conic.
\smallskip
\item The base scheme $F_\phi$ of $\phi$ has degree $10$. The subscheme $F^1_\phi$ is the
  union of a multiple structure on a conic and a quartic curve singular at $p$.
\end{itemize}
\end{pro}
\begin{proof}
First consider the ideal
\[
\mathcal{G}=(f,Q_1)^2\cap(Q_1,Q_2)\cap\II_p^2.
\]
In other words we have
\[
\mathcal{G}=(f^2 Q_2,Q_1^2,f Q_1 z_0,f Q_1z_1,f Q_1 z_2).
\]
From the inclusion $\HH^0(\II_p^2(2H))\subset \CC[z_0,z_1,z_2]$ we consider
$Q_2$ as a homogeneous polynomial of degree $2$ in three variables. Therefore
the image of the rational map
$$\vert\mathcal{G}(4H)\vert\colon
\begin{array}[t]{lll}
  \PP_3& \dashrightarrow &\PP_4 \\
  (z_0:z_1:z_2:z_3) & \mapsto & (Z_0:Z_1:Z_2:Z_3:Z_4)=(f Q_1z_0:f Q_1z_1:f Q_1z_2:f^2 Q_2:Q_1^2)
\end{array}
$$
is the
quadric of $\PP_4$ of equation $Q_2(Z_0,Z_1,Z_2)-Z_3Z_4$ because we have by
homogeneity of $Q_2$
\[
Q_2(f Q_1z_0,f Q_1z_1,f Q_1z_2)=(f Q_1)^2 Q_2(z_0,z_1,z_2).
\]
We have  $\II=\mathcal{G}\cap\II_{p_1}$ where $\II_{p_1}$ is the ideal of a general
point $p_1$ of $\PP_3$. Hence the map $\phi$ factors through the projection from
a point of a quadric of $\PP_4$ so it is birational with base scheme defined by $\II$.
Let us compute the degree of $\phi^{-1}$. Let $H'$ be a general hyperplane of
$\PP'_3$, then $\phi^{-1}(H')$ is a quartic with a double conic. With
Notation~\ref{notaC1C2} the curve $\Cdeux$ is the union of a
component  supported on the conic of ideal $(f,Q_1)$ and the singular quartic curve of ideal
$(Q_1,Q_2)$. The multiple structure defined by $C_2$ on the conic has degree
$8$, as a result $\deg \Cdeux=12$ and $\deg \Cun=4$, then $\phi\in
\Bir_{4,4}(\PP_3,\PP'_3)$.
On another hand, the multiple structure defined by $\II$ on the conic has
degree $6$, so $\deg F_\phi=~10$.
\end{proof}
\begin{definition}
Birational maps described in Proposition \ref{pro:c44} form the family $\mathcal{C}_{4,4}$.
\end{definition}
\begin{cor}\label{C44irreddim}
The family $\mathcal{C}_{4,4}\subset\mathrm{Bir}_{4,4}(\PP_3,\PP_3')$ is irreducible of dimension $37$.
\end{cor}
\begin{proof}
The ideal $\II$ is determined by the singular quartic curve and the conic, but
these two  curves must be on a unique quadric defined by $Q_1$. Thus to
construct $\II$ we need a general quadric, then a singular biquadratic on this
quadric, a general plane and the isolated point $p_1$. It gives an irreducible parameter
space of dimension $9+7+3+3$ to obtain $\II$, so
\[
\dim(\mathcal{C}_{4,4})=22+\dim(\mathrm{PGL}_4)=37.
\]
\end{proof}
\subsection{Proof of Theorem \ref{thm:main}}
From Proposition~\ref{compR44}, Corollaries~\ref{compJ44}, \ref{D44comp} and \ref{C44irreddim}
one has:
\smallskip
\begin{itemize}
\item the closure of $\mathcal{J}_{4,4}$ is an irreducible component of $\Bir_{4,4}(\PP_3,\PP'_3)$,
\item the closure of  $\mathcal{R}_{4,4}$ is an irreducible component of $\Bir_{4,4}(\PP_3,\PP'_3)$,
\item the closure of $\Deter$ is an irreducible component of $\Bir_{4,4}(\PP_3,\PP'_3)$,
\item the family $\mathcal{C}_{4,4}$ is irreducible.
\end{itemize}
\medskip
Note that our definitions of $\mathcal{R}_{4,4}$, $\mathcal{D}_{4,4}$,
$\mathcal{C}_{4,4}$ contain general position assumptions so we were able to
compute degrees and dimensions for any element of these families
(Proposition~\ref{pro:ruled}, Proposition~\ref{propGP3},
Proposition~\ref{pro:c44}). With notation of Definition~\ref{def:semi} we
have:
If $\phi$ is any element of  $\mathcal{J}_{4,4}$, then by
Definition~\ref{def:Jonquieres} $\deg F_\phi^1=12$, but this
degree is $10$ for an element of $\mathcal{C}_{4,4}$. So
$\mathcal{C}_{4,4}$ is not in the closure of
$\mathcal{J}_{4,4}$ by Corollary~\ref{cor:scs}. Moreover  $\dim \SingF_{\phi} = 0$ but it is $1$
for any element of $\mathcal{C}_{4,4}$ hence by
Corollary~\ref{cor:scs} $\mathcal{J}_{4,4}$ is not in the closure of $\mathcal{C}_{4,4}$.
\begin{itemize}
\item If $\phi$ is any element of $\Deter$, then $\deg F_\phi=11$,
$\dim \SingF_\phi=1$ and $\deg \SingF_\phi=1$.
\item If $\phi$ is any element of
$\mathcal{C}_{4,4}$, then $\deg F_\phi=10$, $\dim \SingF_\phi=1$ and $\deg
\SingF_\phi=2$.
\item If $\phi$ is any element of  $\mathcal{R}_{4,4}$, then $\deg
F_\phi=9$, $\dim \SingF_\phi=1$ and $\deg \SingF_\phi=3$.
\end{itemize}
By Corollary~\ref{cor:scs}, since the map $\alpha$ and $\eta$  cannot decrease by
specialization one gets that none of the family $\mathcal{R}_{4,4}$,
$\mathcal{C}_{4,4}$, $\Deter$, $\mathcal{J}_{4,4}$ is in the closure of another
one.
Let $H$ (resp. $H'$) be a general plane of $\PP_3$ (resp. $\PP'_3$). The quartic
surface $\phi^{-1}(H')$ has the following property
\begin{center}
\begin{tabular}[c]{|c|c|c|c|c|}
  \hline
$\phi$ general in  & $\mathcal{R}_{4,4}$ & $\mathcal{C}_{4,4}$ & $\mathcal{D}_{4,4}$ & $\mathcal{J}_{4,4}$
  \\ \hline
$\phi^{-1}(H')$ & has a triple line & has a double conic& has a double line& is
normal \\ \hline
\end{tabular}
\end{center}
It gives the genus of $H\cap \phi^{-1}(H')$.\hfill$\square$

\bibliographystyle{plain}
\bibliography{biblio}

\end{document}